\documentclass[journal]{IEEEtran}

\usepackage[noadjust]{cite}

\usepackage[english]{babel}
\usepackage[final,pdftex]{graphicx}

\usepackage[cmex10]{amsmath}
\usepackage{amsfonts}
\usepackage{amssymb}
\usepackage{amsthm}
\usepackage{mathbbol}	
\usepackage{steinmetz}	

\usepackage[T1]{fontenc}

\newtheorem{theorem}{Theorem}
\newtheorem{corollary}[theorem]{Corollary}

\newtheorem*{remark}{Remark}

\newtheorem*{example}{Example}
\newtheorem{lemma}{Lemma}[section]

\usepackage{enumerate}
\usepackage{array}

\usepackage{url}

\newlength{\IEEECOLWIDTH}
\usepackage{booktabs}
\newcommand{\ra}[1]{\renewcommand{\arraystretch}{#1}}

\def\TITLE{On the existence and linear approximation of the power flow solution in power distribution networks}
\def\TITLEnl{On the existence and linear approximation of the power flow solution in power distribution networks}
\def\AUTHORS{Saverio Bolognani and Sandro Zampieri}

\usepackage[hyperindex=true, %
  pdftitle={\TITLE},%
  pdfauthor={\AUTHORS},%
  colorlinks=true,%
  pagebackref=false,%
  plainpages=false,%
  pdfpagelabels,%
  linkcolor=NavyBlue,%
  citecolor=NavyBlue,%
  filecolor=NavyBlue,%
  urlcolor=NavyBlue%
]{hyperref} 

\usepackage[absolute,overlay]{textpos}
\usepackage[dvipsnames]{xcolor}

\def\vnodes{{\mathcal{V}}}
\def\lnodes{{\mathcal{L}}}
\def\pnodes{{\mathcal{Q}}}
\def\nonodes{{n}}
\def\node{{h}}
\def\nodealt{{k}}

\def\X{{Z}}

\def\C{\mathbb{C}}
\def\R{\mathbb{R}}
\def\1{\mathbb{1}}

\DeclareMathOperator{\diag}{diag}

\def\slackangle{{\theta_0}}

\DeclareMathOperator{\realpart}{Re}
\DeclareMathOperator{\imagpart}{Im}

\begin{document}

\title{\TITLEnl}

\author{\AUTHORS
\thanks{This work is supported by European
    Community 7th Framework Programme under grant agreement
    n. 257462 HYCON2 Network of Excellence.}%
\thanks{S.~Bolognani is with the Laboratory on Information and Decision Systems,
	Massachusetts Institute of Technology, Cambridge (MA), USA. Email: \texttt{\small saverio@mit.edu}.}%
\thanks{S.~Zampieri is with the Department of Information Engineering,
	University of Padova, Padova (PD), Italy. Email \texttt{\small zampi@dei.unipd.it}}%
\thanks{\textcopyright 2016 IEEE.  Personal use of this material is permitted.  Permission from IEEE must be obtained for all other uses, in any current or future media, including reprinting/republishing this material for advertising or promotional purposes, creating new collective works, for resale or redistribution to servers or lists, or reuse of any copyrighted component of this work in other works.}%
}

\maketitle

\begin{textblock*}{\textwidth}(15mm,50mm) 
\centering \bf \textcolor{NavyBlue}{Published on \emph{IEEE Transactions on Power Systems,} vol. 31, no. 1, pp. 163--172, Jan. 2016.\\\url{https://doi.org/10.1109/TPWRS.2015.2395452}}
\end{textblock*}

\begin{abstract}
We consider the problem of deriving an explicit approximate solution of the nonlinear power equations that describe a balanced power distribution network.
We give sufficient conditions for the existence of a practical solution to the power flow equations, and we propose an approximation that is linear in the active and reactive power demands of the PQ buses. For this approximation, which is valid for generic power line impedances and grid topology, we derive a bound on the approximation error as a function of the grid parameters. We illustrate the quality of the approximation via simulations, we show how it can also model the presence of voltage controlled (PV) buses, and we discuss how it generalizes the DC power flow model to lossy networks.
\end{abstract}

\begin{IEEEkeywords}
Power systems modeling, load flow analysis, power distribution networks, fixed point theorem.
\end{IEEEkeywords}

\section{Introduction}

The problem of solving the power flow equations that describe a power system, i.e. computing the steady state of the grid (typically the bus voltages) given the state of power generators and loads, is among the most classical tasks in circuit and power system theory.
An analytic solution of the power flow equations is typically not available, given their nonlinear nature. For this reason, notable effort has been devoted to the design of numerical methods to solve systems of power flow equations, to be used both in offline analysis of a grid and in real time monitoring and control of the system (see for example the review in \cite{Eminoglu2005}).

Specific tools have been derived for the approximate solution of such equations, based on some assumptions on the grid parameters. In particular, if the power lines are mostly inductive, equations relating active power flows and bus voltage angles result to be approximately linear, and decoupled from the reactive power flow equations, resulting in the \emph{DC power flow model} (see the review in \cite{Stott2009}, and the more recent discussion in \cite{Dorfler2013}).

We focus here on a specific scenario, which is the power distribution grid. More specifically, we are considering a balanced medium voltage grid which is connected to the power transmission grid in one point (the distribution substation, or PCC, point of common coupling), and which hosts loads and possibly also microgenerators.

Power distribution grids have recently been the object of an unprecedented interest.
Its operation has become more challenging since the deployment of distributed microgeneration and the appearance of larger constant-power loads 
(electric vehicles in particular).
These challenges motivated the deployment of ICT (information and communications technology) in the power distribution grid, in the form of sensing, communication, and control devices, in order to operate the grid more efficiently, safely, reliably, and within the its voltage and power constraints.
These applications have been reviewed in \cite{Hill2012}, and include real-time feedback control \cite{Vovos2007,Bolognani2013,Zhang2013}, automatic reconfiguration \cite{Baran1989a,DallAnese2014}, and load scheduling \cite{Deilami2011,Sundstroem2012}.
In order to design the control and optimization algorithms for these applications, an analytic (rather than numerical) solution of the power flow equations would be extremely convenient.
Unfortunately, because in the medium voltage grid the power lines are not purely inductive, power flow equations include both the active and reactive power injection/demands, and both the voltage angles and magnitudes, in an entangled way.
The explicit DC power flow model therefore does not apply well.

The contribution of this paper is twofold.
First, we give sufficient conditions for the existence of a practical solution of the nonlinear power flow equations in power distribution networks (Theorem~\ref{theo:taylor_expansion_complex}).
Second, we derive a tractable approximate solution to the power flow problem, linear in the complex power injections, providing a bound on the approximation error (Corollary~\ref{cor:approx}).

In the remainder of this section, we review relevant related works. 
In Section~\ref{sec:model}, we present the nonlinear equations that define the power flow problem. In Section~\ref{sec:main} we present our main existence result, together with the linear approximate solution. We illustrate such approximation via simulations in Section~\ref{sec:simulations}, and we compare it with the classical DC power flow model in Section~\ref{sec:dcpower}, where we also show how to incorporate voltage controlled (PV) buses.

\subsection{Related works}

Conditions have been derived in order to guarantee the existence of a solution to the power flow equations in the scenario of a grid of nonlinear loads. 

Many results are based on the degree theory \cite{Wu1974,Ohtsuki1977,Chua1977}.
In \cite{Pai1999}, for example, exponential model is adopted for the loads, and sufficient conditions for the existence of a solution are derived. These conditions are however quite restrictive, and do not include constant power (PQ) buses.
In \cite{Baran1990}, on the other hand, the existence of a solution is proved by exploiting the radial structure of the grid, via an iterative procedure which is closely related to a class of  iterative numerical methods specialized for the power distribution networks \cite{Kersting1984,Ghosh1999}.

The existence of solutions to the power flow equations has been also studied in order to characterize the security region of a grid, i.e. the set of power injections and demands that yield acceptable voltages across the network. These results include \cite{Wu1982}, and others where however the decoupling between active and reactive power flows is assumed \cite{Ilic1992}. Other works in which the DC power flow assumption plays a key role are \cite{Arapostathis1981} and \cite{Baillieul1982}, both focused on active power flows across the grid. On the other hand, the results in \cite{Thorp1986} focus on the reactive power flows and on the voltage magnitudes at the buses.

In \cite{Molzahn2013} the implicit function theorem is used in order to advocate the existence of a power flow solution, without providing an approximate expression for that.

It is worth noticing that the linear approximate model that we are presenting in this paper shares some similarities with the method of \emph{power distribution factors} \cite{Wood1996}, which allow to express variations in the state (voltage angles) as a linear function of active power perturbations. This method is also typically based on the DC power flow assumptions, even if a formulation in rectangular coordinates (therefore modeling reactive power flows) has been proposed in \cite{Grijalva2002}. Notice that, except for the seminal works on power distribution factors \cite{Sauer1981,Ng1981}, and the more recent results in \cite{Baldick2003}, most of the related results consists in algorithms that allow to compute this factors only numerically, from the Jacobian of the power flow equations.

The approximate power flow solution proposed in this paper has been presented in a preliminary form in \cite{Bolognani2013,BG-JWSP-FD-SZ-FB:13a}, where however no guarantees on the existence of such solution and on the quality of the approximation were given.

\section{Power flow equations}
\label{sec:model}

We are considering a portion of a symmetric and balanced power distribution network, connected to the grid at one point, delivering power to a number of buses, each one hosting loads and possibly also microgenerators.
We denote by $\{0,1,\ldots,\nonodes\}$ the set of buses, where the index $0$ refers to the PCC.

We limit our study to the steady state behavior of the system,
when all voltages and currents are sinusoidal signals at the same frequency.
Each signal can therefore be represented via a complex number $y = |y| e^{j\angle y}$ whose absolute value $|y|$ corresponds to the signal
root-mean-square value, and whose phase $\angle y$ corresponds to the phase of the signal with respect to an arbitrary global reference.

In this notation, the steady state of the network is described by the voltage $v_\node \in \C$ and by the injected current $i_\node \in \C$ at each node $\node$.
We define the vectors $v, i \in \C^{n+1}$, with entries $v_\node$ and $i_\node$, respectively.

Each bus $\node$ of the network is characterized by a law 
relating its injected current $i_\node$ with its voltage $v_\node$.
We model bus $0$ as a \emph{slack node}, in which a voltage is imposed
\begin{equation}
v_0 = V_0 e^{j \slackangle},
\label{eq:PCCmodel}
\end{equation}
where $V_0, \slackangle \in\R$ are such that $V_0 \ge 0$ and $-\pi < \slackangle \le \pi$. 
We model all the other nodes as PQ buses, in which the injected complex power (active and reactive powers) is imposed and does not depend on the bus voltage. This model describes the steady state of most loads, and also the behavior of microgenerators, that are typically connected to the grid via power inverters \cite{Green2007}. According to the PQ model, we have that, at every bus,
\begin{equation}
v_\node\bar i_\node= s_\node\,\quad \forall \node \in \lnodes :=\{1,\ldots,n\},
\label{eq:ZIPModel}
\end{equation}
where $s_\node$ is the imposed complex power.

A more compact way to write these nonlinear power flow equations is the following. 
Let the vectors $i_\lnodes, v_\lnodes, s_\lnodes$ be vectors in $\mathbb C^n$ having $i_\node, v_\node, s_\node$, $\node\in\lnodes$ as entries. Then we have
\begin{equation}
\begin{cases}
v_0 = V_0 e^{j \slackangle}\\
s_\lnodes = \diag(\bar i_\lnodes)v_\lnodes
\end{cases}
\label{eq:nonlineareqs}
\end{equation}
where $\bar i_\lnodes$ is the vector whose entries are the complex conjugates of the entries of $i_\lnodes$ and where $\diag(\cdot)$ denotes a diagonal matrix having the entries of the vector as diagonal elements.

We model the grid power lines via their nodal admittance matrix $Y \in \C^{(n+1)\times(n+1)}$, which gives a linear relation between bus voltages and currents, in the form 
\begin{equation}
i=Yv.
\label{eq:admittancematrix}
\end{equation}

In the rest of the paper, we assume that the shunt admittances at the buses are negligible. Therefore the nodal admittance matrix satisfies
\begin{equation}
Y \1 = 0,
\label{eq:Ykernel}
\end{equation}
where $\1$ is the vector of all ones. Under this assumption, the matrix $Y$ corresponds to the weighted Laplacian of the graph describing the grid, with edge weights equal to the admittance of the corresponding power lines.
We will show in a remark in Section~\ref{sec:main} how this assumption can be relaxed, so that the entire analysis can be extended seamlessly to the case in which shunt admittances of the lines are not negligible.

Considering the same partitioning of the vectors $i,v$ as before, we can partition the admittance matrix $Y$ accordingly, and rewrite \eqref{eq:admittancematrix} as
$$
\left[\begin{array}{c} i_0\\i_\lnodes\end{array}\right]=
\left[\begin{array}{cc} Y_{0 0}&Y_{0 \lnodes}\\Y_{\lnodes 0}&Y_{\lnodes \lnodes}\end{array}\right]
\left[\begin{array}{c} v_0\\v_\lnodes\end{array}\right].
$$
where $Y_{\lnodes\lnodes}$ is invertible because, if the graph representing the grid is a connected graph, then $\1$ is the only vector in the null space of $Y$ \cite{Godsil2001}.
Using \eqref{eq:Ykernel} we then obtain
\begin{equation}
v_\lnodes = v_0 \1 + \X i_\lnodes 
\label{eq:lineareqs}
\end{equation}
where the impedance matrix $\X \in\C^{n\times n}$ is defined as
\begin{equation*}
\X :=Y_{\lnodes \lnodes}^{-1},
\end{equation*}

Objective of the power flow analysis is to determine 
from these equations 
the voltages $v_\node$ and the currents $i_\node$  as functions of $V_0,\slackangle$ and $s_1,\ldots,s_{n}$, namely
\begin{align*}
v_\node=&v_\node(V_0,\slackangle,s_1,\ldots,s_{n})\\
i_\node=&i_\node(V_0,\slackangle,s_1,\ldots,s_{n}).
\end{align*}
In general, because of the nonlinear nature of the loads, we may have no solution or more than one solution for fixed $V_0,\slackangle$ and $s_\node$, as the following simple example shows.

\begin{example}[Two-bus case]
Consider the simplest grid made by two nodes, node $0$ being the slack bus (where we let $\slackangle=0$), and node $1$ being a PQ bus. In this case we have that the following equations have to be satisfied
\begin{equation}
\begin{cases}
v_1 \bar i_1=s_1\\
v_1  = V_0 + \X_{11} i_1
\end{cases}
\label{eq:example}
\end{equation}
Assume
that $\X_{11}=1$, and that $s_1$ is real.
The system of equations \eqref{eq:example} can then be solved analytically. In fact it can be found that if $V_0^2 + 4s_1<0$ there are no solutions. When on the contrary $V_0^2 + 4s_1>0$, there are two distinct solutions
$$
i_{1} = \frac{-V_0 \pm \sqrt{V_0^2 + 4s_1}}{2}.
$$
Notice that, if $V_0$ is large, then the solutions exist and, since $\sqrt{V_0^2+4s_1} =V_0\sqrt{1+4s_1/V_0^2}\simeq V_0(1+2s_1/V_0^2)=V_0 + 2s_1/V_0$, the current $i_1$ take the two values
\begin{align*}
i_1^+ \simeq s_1/V_0, \qquad
i_1^- \simeq - V_0 - \frac{s_1}{V_0}.
\end{align*}
Therefore, when $V_0$ is large, one solution consists in small currents (thus small power losses and voltage close to the nominal voltage across all the network), while the other consists in larger currents, larger power losses, and larger voltage drops.
Of course, the system should be controlled so that it works at the first working point.
\end{example}

The intuition from this simple example is developed in the next section, where the existence and uniqueness of a practical solution to the power flow equations (i.e. a solution at which the grid can practically and reliably be operated) is studied, and an approximate power flow solution (linear in the power terms) is proposed.

\section{Main result}
\label{sec:main}

Define 
$$
f := v_0 \bar i_\lnodes - s_\lnodes = V_0 e^{j\slackangle} \bar i_\lnodes - s_\lnodes,
$$
so that we have
\begin{equation}
i_\lnodes = \frac{1}{\bar v_0}(\bar f+ \bar s_\lnodes)
=\frac{e^{j\slackangle}}{V_0}(\bar f+\bar s_\lnodes).
\label{eq:itof}
\end{equation}

By putting together \eqref{eq:itof} with \eqref{eq:nonlineareqs} and \eqref{eq:lineareqs},
we get
\begin{align*}
s_\lnodes
&= \diag(\bar i_\lnodes)v_\lnodes\\
&= \frac{e^{-j\slackangle}}{V_0} \diag(f+s_\lnodes) 
\left[ e^{j\slackangle}V_0 \1 + \frac{e^{j\slackangle}}{V_0} \X (\bar f+\bar s_\lnodes)\right]\\
&= f+s_\lnodes+\frac{1}{V_0^2}\diag( f+s_\lnodes) \X (\bar f+\bar s_\lnodes),
\end{align*}
and therefore
\begin{equation}
f = -\frac{1}{V_0^2} \diag(f+ s_\lnodes) \X (\bar f+\bar s_\lnodes).
\label{eq:fequation}
\end{equation}

We can determine a ball where there exists a unique solution $f$ to this equation by applying the Banach fixed point theorem \cite{Debnath2005}. 
In order to do so, define the function
$$
G(f):=-\frac{1}{V_0^2}\diag(f+ s_\lnodes) \X (\bar f+\bar s_\lnodes).
$$

Consider the standard 2-norm $\|\cdot\|$ on $\C^n$ defined as
$$
\|x\| := \sqrt{\sum_\node |x_\node|^2}.
$$
Let us then define the following matrix norm\footnotemark  on $\C^{n \times n}$
\begin{align}
\|A\|^* &:= \max_{\node} \|A_{\node\bullet}\|=\max_{\node}\sqrt{\sum_\nodealt |A_{\node\nodealt}|^2}
\label{eq:matrix2norm}
\end{align}
where the notation $A_{\node\bullet}$ stands for the $\node$-th row of $A$.

\footnotetext{The $*$ sign indicates that we are not referring to the norm induced by the  vector norm.}

The following result holds.

\begin{theorem}[Existence of a practical power flow solution]
Consider the vector 2-norm $\|\cdot\|$ on $\C^{n}$,
and the matrix norm $\|\cdot\|^*$ defined in \eqref{eq:matrix2norm}.
If
\begin{equation}\label{eq:disf}
V_0^2 > 4 \|\X\|^* \|s_\lnodes\|
\end{equation}
then there exists a unique solution $v_\lnodes$ of the power flow equations \eqref{eq:nonlineareqs} and  \eqref{eq:lineareqs} in the form
\begin{equation}\label{eq:exp}
v_\lnodes
=
V_0 e^{j\slackangle}
\left(\1+\frac{1}{V_0^2} \X \bar s_\lnodes + \frac{1}{V_0^4} \X \lambda\right)
\end{equation}
where $\lambda \in \C^n$ is such that
\begin{equation}\label{eq:err}
\|\lambda\| \le 4 \|\X\|^* \|s_\lnodes\|^2.
\end{equation}
\label{theo:taylor_expansion_complex}
\end{theorem}

\begin{proof}
Let 
\begin{equation}\label{eq:delta}
\delta:=\frac{4 \|\X\|^* }{V_0^2}\|s_\lnodes\|^2
\end{equation}
and $B:=\{f\in\C^n\ |\ \|f\| \le \delta\}$. 
In order to apply the Banach fixed point theorem, we need to show that, under the hypotheses of the theorem,
\begin{align}
& G(f) \in B\text{ for all $f\in B$} \label{eq:banach1} \\
& \|G(f')-G(f'')\| \le k \|f'-f''\| \text{ for all $f',f''\in B$} \label{eq:banach2}
\end{align}
for a suitable constant $0\le k<1$.
We prove first \eqref{eq:banach1}. Observe that, by using Lemma~\ref{lem:pnorm} in the case $p=2$, we have
\begin{align*}
\|G(f)\|
& \le \frac{1}{V_0^2} \|\X\|^* \ \|f+s_\lnodes\|^2\\
&\le\frac{1}{V_0^2} \|\X\|^* \ (\|f\|+\|s_\lnodes\|)^2\\
& \le \frac{1}{V_0^2} \|\X\|^* \ \left(\delta +\|s_\lnodes\|\right)^2,
\end{align*}
where we used the fact that $\|f\| \le\delta$.
Now, using the definition \eqref{eq:delta} of $\delta$ and Lemma~\ref{lemma} in the Appendix (with $a=\|\X\|^*/V_0^2$ and $b=\|s_\lnodes\|$)
we can argue that, if \eqref{eq:disf} is true, then
$$
\|G(f)\| \le \frac{1}{V_0^2} \|\X\|^* \left(\delta+\|s_\lnodes\|\right)^2
\le \delta.
$$

We prove now \eqref{eq:banach2}. 
It is enough to notice that, by applying Lemma~\ref{lem:F} in the Appendix (with $A=-\X/V_0^2$, $x=f$ and $a=s_\lnodes$) we obtain that 
\begin{align*}
& \|G(f')-G(f'')\| \\
& \qquad \le \frac{1}{V_0^2} \|\X\|^* \left(\|f'+f''\|+2\|s_\lnodes\|\right) \|f'-f''\|\\
& \qquad \le \frac{2}{V_0^2} \|\X\|^* \left( \delta +\|s_\lnodes\|\right)\|f'-f''\|\\
& \qquad = k\|f'-f''\|
\end{align*}
where
$$
k:=\frac{2}{V_0^2} \|\X\|^* \left(\frac{4 \|\X\|^* }{V_0^2}\|s_\lnodes\|^2+\|s_\lnodes\|\right).
$$
Finally, notice that by using \eqref{eq:disf} we obtain
\begin{align*}
k 
& <\frac{2}{V_0^2} \|\X\|^* \frac{V_0^2}{4 \|\X\|^* } \left(\frac{4 \|\X\|^* }{V_0^2}\frac{V_0^2}{4 \|\X\|^* } +1\right)=1.
\end{align*}

Now from \eqref{eq:banach1} and \eqref{eq:banach2}, by applying the Banach fixed point theorem, we can argue that there exists unique solution $f\in B$ of equation \eqref{eq:fequation}.
Then by using \eqref{eq:lineareqs} and \eqref{eq:itof} we have that
\begin{align*}
v_\lnodes
&= V_0 e^{j\slackangle} \1  + \X \frac{e^{j\slackangle}}{V_0}(\bar s_\lnodes+\bar f) \\
& = V_0 e^{j\slackangle} \left(\1 +\frac{1}{V_0^2} \X \bar s_\lnodes+\frac{1}{V_0^2} \X \bar f\right).
\end{align*}
In order to prove \eqref{eq:err} it is enough to define
$\lambda:=V_0^2 \bar f$.
\end{proof}

\begin{remark}
The norm $\|\X\|^*$ can be put in direct relation with the induced matrix 2-norm $\|\X\|$, and with structural properties of the graph that describes the power grid. Indeed
$$
\|\X\|^* = \max_\node \|e_\node^T \X\| \le 
\max_{\|v\| = 1}\|\X^T v\| = \|\X\|,
$$
where $e_\node$ is the $\node$-th vector of the canonical base.
It can be shown that 
$\left[\begin{smallmatrix} 0 & 0 \\ 0 & \X \end{smallmatrix}\right]$ 
is one possible pseudoinverse of the weighted Laplacian $Y$ of the grid.
Therefore we have that
$$
\|\X\|^* \le \frac{1}{\sigma_2(Y)},
$$
where $\sigma_2(Y)$ is the second smallest singular value of $Y$ (the smallest one being zero).
In the special case in which all the power lines have the same $X/R$ ratio (i.e. their impedances have the same angle, but different magnitudes), then $\sigma_2(Y)$ corresponds also to the second smallest eigenvalue of the Laplacian, which is a well known metric for the graph connectivity.
Given this relation between $\|\X\|^*$ and $\|\X\|$, the assumption \eqref{eq:disf} in Theorem~\ref{theo:taylor_expansion_complex} is satisfied if
$$
V_0^2 > 4 \frac{\|s_\lnodes\|}{\sigma_2(Y)}.
$$
This condition resembles similar results that have been proposed for example in \cite{Ilic1992} for the analysis of the feasibility of the power flow problem.
Interestingly, the role of similar spectral connectivity measures of the grid has been recently investigated also for grid synchronization and resilience problems (see the discussion in \cite{Doerfler2013} and 
\cite{Pagani2013}, respectively).
\end{remark}

\begin{corollary}[Approximate power flow solution]
\label{cor:approx}
Consider the vector 2-norm $\|\cdot\|$ on $\C^{n}$,
and the matrix norm $\|\cdot\|^*$ defined in \eqref{eq:matrix2norm}.
Let the assumption of Theorem~\ref{theo:taylor_expansion_complex} be satisfied. Then the solution $v_\lnodes$ of the power flow nonlinear equations is approximated by 
\begin{equation}
\hat v_\lnodes
:=
V_0 e^{j\slackangle}
\left(\1+\frac{1}{V_0^2} \X \bar s_\lnodes\right),
\label{eq:approximate}
\end{equation}
and the approximation error satisfies 
\begin{equation}
|v_\node - \hat v_\node| \le 
\frac{4}{V_0^3} \|\X_{\node\bullet}\| \|\X\|^* \|s_\lnodes\|^2,
\label{eq:buserror2}
\end{equation}
where, as before, $\X_{\node\bullet}$ is the $\node$-th row of $\X$.
\end{corollary}
\begin{proof}
We have, from \eqref{eq:exp}, for any bus $\node \in \lnodes$,
$$
|v_\node - \hat v_\node| 
= \frac{1}{V_0^3} \left| \X_{\node\bullet} \lambda \right|,
$$
By using Cauchy-Schwarz inequality and the bound \eqref{eq:err} we obtain \eqref{eq:buserror2}.
\end{proof}

Theorem~\ref{theo:taylor_expansion_complex} and Corollary~\ref{cor:approx} can be interpreted as a way to model the grid as a linear relation between the variables $v$ and $s$.
In fact, power flow equations already contain the simple linear relation \eqref{eq:admittancematrix} between the variables $i$ and $v$, together with  an implicit nonlinear relation between $v$ and $s$. The previous results say that, in case $s$ is sufficiently small, there is a way to make the relation between $v$ and $s$ explicit (Theorem~\ref{theo:taylor_expansion_complex}), 
and to find a linear approximated relation between these variables
(Corollary~\ref{cor:approx}).
This interpretation will be further elaborated in Section~\ref{sec:dcpower}, in order to extend the model to grids where voltage regulated (PV) nodes are present.

Notice moreover that the approximate model~\eqref{eq:approximate} can be manipulated, via premultiplication by the admittance matrix $Y$, in order to obtain the sparse linear equation
\begin{equation}
Y \hat v = \frac{e^{j\theta_0}}{V_0} \bar s,
\label{eq:sparse}
\end{equation}
where $\hat v$ and $s$ are the augmented vectors
$\left[\begin{smallmatrix}\hat v_0 \\ \hat v_\lnodes\end{smallmatrix}\right]$
and
$\left[\begin{smallmatrix}-\1^T s_\lnodes \\ s_\lnodes\end{smallmatrix}\right]$, respectively.
The resemblance of \eqref{eq:sparse} with the DC power flow model will be investigated in Section~\ref{sec:dcpower}, where the proposed approximate solution is presented in polar coordinates.

\begin{remark}[Non-zero shunt admittances] The grid model \eqref{eq:admittancematrix} in which the matrix $Y$ is assumed to satisfy \eqref{eq:Ykernel} is based on the assumption of zero nodal shunt admittances. In the case in which shunt admittances are not negligible, the proposed analysis can be modified accordingly. 
In this more general case the matrix $Y$ is invertible, and for almost all grid parameters of practical interest, the submatrix $Y_{\lnodes \lnodes}$ is also invertible.
Equation \eqref{eq:lineareqs} becomes 
$
v_\lnodes =  v_0 w + \X i_\lnodes, 
$
where $\X :=Y_{\lnodes \lnodes}^{-1}$, and where 
$
w:=-Y_{\lnodes \lnodes}^{-1}Y_{\lnodes 0}\in\C^{n}
$
is a perturbation of the vector $\1$ and corresponds to the normalized no-load voltage profile of the grid.

The reasoning of Theorem \ref{theo:taylor_expansion_complex} can be repeated by defining 
$f := V_0 e^{j\slackangle} W \bar i_\lnodes - s_\lnodes$,
where the diagonal matrix
$
W := \diag(w)
$
is a perturbation of the identity matrix.
In this case, $f$ has to satisfy the equation $f = G(f)$ where
$$
G(f):=-\frac{1}{V_0^2}\diag(f+ s_\lnodes) W^{-1} \X \bar W^{-1} (\bar f+\bar s_\lnodes).
$$

Condition \eqref{eq:disf} for the existence of a power flow solution is then replaced by the condition
$$
V_0^2 > 4 \|W^{-1} \X \bar W^{-1}\|^* \|s_\lnodes\|,
$$
and the approximate power flow solution in Corollary~\ref{cor:approx} becomes
$$
\hat v_\lnodes
:=
V_0 e^{j\slackangle} \left(w+\frac{1}{V_0^2} \X \bar W^{-1} \bar s_\lnodes\right).
$$
\end{remark}

The results of Theorem~\ref{theo:taylor_expansion_complex} and Corollary~\ref{cor:approx} hold also for other vector norms, different from the vector 2-norm, as the following remarks show.

\begin{remark}[$p$-norms]
Consider the standard p-norm $\|\cdot\|_p$ on $\C^{n}$ defined as
\begin{equation}\label{eq:vectorpnorm}
\|x\|_p := \left(\sum_\node |x_\node|^p\right)^{1/p}.
\end{equation}
if $1\le p<\infty$ and by
$$
\|x\|_{\infty} := \max_\node |x_\node |.
$$

Let us then define the following matrix norm on $\C^{n \times n}$
\begin{equation}\label{eq:matrixpnorm}
\|A\|^*_p := \max_{\node} \|A_{\node\bullet}\|_p
\end{equation}
where, as before, the notation $A_{\node\bullet}$ denotes the $\node$-th row of $A$.

We have that Theorem~\ref{theo:taylor_expansion_complex} holds with respect to the vector $p$-norm $\|\cdot\|_p$ on $\C^{n}$ and 
the matrix norm $\|\cdot\|^*_q$ on $\C^{n \times n}$
where $p,q \in [1,\infty) \cup \{\infty\}$ are such that $1/p+1/q=1$.
This follows from Lemma~\ref{lem:pnorm}.
The sufficient condition \eqref{eq:disf} becomes
$$
V_0^2 > 4 \|\X\|_q^* \|s_\lnodes\|_p
$$
As before, the statement of Theorem~\ref{theo:taylor_expansion_complex} allows to derive a bound on the error of the approximate power flow solution proposed in Corollary~\ref{cor:approx} (this time using H\"older's inequality), replacing \eqref{eq:buserror2} with
\begin{equation*}
|v_\node - \hat v_\node| \le 
\frac{4}{V_0^3} \|\X_{\node\bullet}\|_q \|\X\|_q^* \|s_\lnodes\|_p^2.
\end{equation*}

The particular case $p=1$, $q=\infty$ has an interesting physical interpretation. It can be inferred from \eqref{eq:lineareqs} that, as power line impedances have positive resistance and reactance,
$$
|\X_{\node\node}| \ge |\X_{\node\nodealt}| \quad \text{for all $\node,\nodealt \in \lnodes$}.
$$
Therefore $\|\X_{\node\bullet}\|_\infty = |\X_{\node\node}|$, which, 
in a radial network, is the length of the path (in terms of the magnitude of the path impedance) connecting bus $\node$ to bus $0$. We denote such length by $\ell_\node$.
Then, $\|\X\|_\infty^* = \max_{\node} |\X_{\node\node}|$
corresponds in a radial network to the \emph{breadth of the grid $\ell_\text{max}$}, defined as the length of the longest path connecting a bus $\node$ to bus $0$.
On the other hand, $\|s_\lnodes\|_1$ corresponds to a metric for the \emph{total load} of the grid, defined as the sum $S_\text{tot}$ of the apparent powers $|s_\node|$ at every node $\node \in \lnodes$. 
Therefore, in this case, the sufficient condition for the existence of a practical solution becomes
$$
V_0^2 > 4 \ell_\text{max} S_\text{tot},
$$
and the bound on the approximation error is 
\begin{equation}
|v_\node - \hat v_\node| \le 
\frac{4}{V_0^3} \ell_\node \ell_\text{max} S_\text{tot}^2.
\label{eq:buserror1}
\end{equation}
\end{remark}

In general, Theorem~\ref{theo:taylor_expansion_complex} gives only sufficient conditions for the existence and uniqueness of a solution of the nonlinear power flow equations, and a bound on the error of the approximation \eqref{eq:approximate}. By revisiting the example proposed in Section~\ref{sec:model} we show that both the sufficient condition \eqref{eq:disf} and the bound \eqref{eq:buserror2} can be tight.

\begin{example}[Two-bus case]
Consider the same example considered in Section~\ref{sec:model}.
Notice that, regardless of the vector norm that is chosen,
$$
\|Z\|^* = |Z_{11}| = 1 
\quad \text{and} \quad
\|s\| = |s_1|.
$$
The sufficient condition \eqref{eq:disf} in Theorem~\ref{theo:taylor_expansion_complex} then becomes
$$
V_0^2 > 4 |s_1|.
$$
In the case that we were considering in the example, with $s_1$ real and negative (corresponding to an active power load), it is easy to see that this condition is also a necessary condition for the existence of a solution.
In the special case in which the condition is marginally satisfied ($s_1 = -V_0^2 / 4$), it is possible to compute both the exact solution
$$
v_1 = V_0 + i_1 = \frac{1}{2} V_0, 
$$ 
and the approximate solution according to \eqref{eq:approximate},
$$
\hat v_1 = V_0 \left( 1 + \frac{1}{V_0^2} s_1 \right) = \frac{3}{4} V_0.
$$
The approximation error is clearly $|v_1 - \hat v_1| = V_0/4$,
which marginally satisfies the approximation error bound \eqref{eq:buserror2}, i.e.
$$
|v_1 - \hat v_1| \le \frac{4}{V_0^3} |s_1|^2 = \frac{V_0^2}{4}.
$$
\end{example}

\section{Simulations}
\label{sec:simulations}

\begin{figure*}[t]
\centering
\includegraphics[scale=1.06]{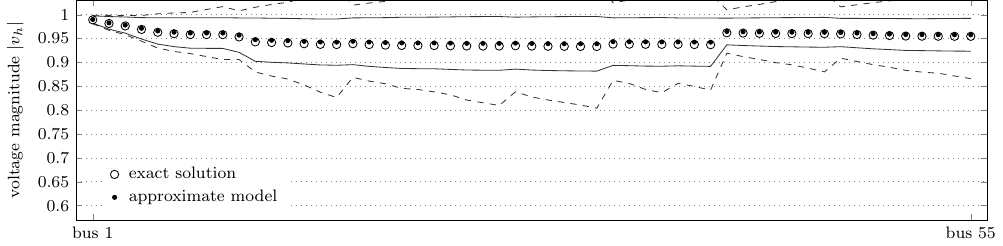}
\caption{Bus voltages in the modified IEEE 123 test feeder \cite{github_approx-pf}. The circles represent the exact solution of the nonlinear power flow equations. The dots represent the approximate solution $\hat v$ given by \eqref{eq:approximate}. The solid line represent the bound \eqref{eq:buserror2} on the approximation error obtained by considering the 2-norm and the corresponding matrix norm $\|\X\|^*$, while the dashed line represents the error bound \eqref{eq:buserror1}, obtained by adopting the norms $\|s_\lnodes\|_1$ and $\|\X\|_\infty^*$.}
\label{fig:voltage}
\end{figure*}

\begin{figure*}[t]
\centering
\includegraphics[scale=1.06]{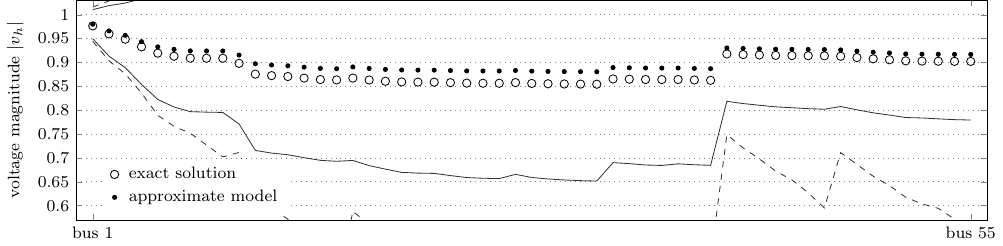}
\caption{Bus voltages in the modified IEEE 123 test feeder \cite{github_approx-pf}, in the case of uniform overload (where all active and reactive demands are doubled). The same convention of Figure~\ref{fig:voltage} is adopted for the approximate solution, the exact solution, and the approximation bounds.}
\label{fig:voltage_overload}
\end{figure*}

\begin{figure*}[t]
\centering
\includegraphics[scale=1.06]{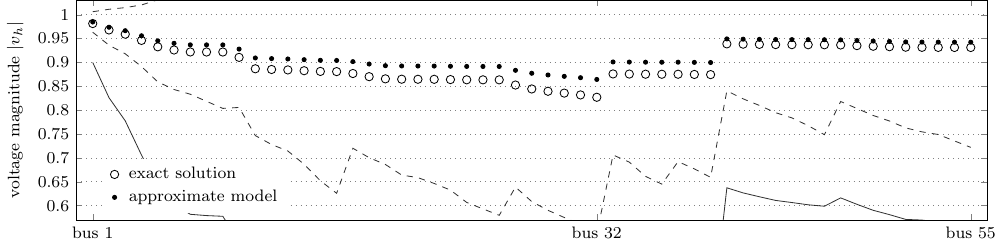}
\caption{Bus voltages in the modified IEEE 123 test feeder \cite{github_approx-pf}, in the case of lumped overload (where the power demand of bus 32 is increased to 2 MW and 1 MVAR). The same convention of Figure~\ref{fig:voltage} is adopted for the approximate solution, the exact solution, and the approximation bounds.}
\label{fig:voltage_spotoverload}
\end{figure*}

In order to illustrate the results presented in the previous section, we considered a symmetric balanced testbed inspired to the IEEE 123 test feeder \cite{Kersting2001}.

The details and the data of the adopted testbed are available online \cite{github_approx-pf}, together with the Matlab / GNU Octave and MatPower source code used for the simulations.

In a per-unit notation, where $V_0$ is taken as the reference nominal voltage, and $S_N = 1$MW is taken as the reference nominal power, we have that
\begin{align*}
& 
\begin{cases}
\|s\| = 0.7015 \\
\|\X\|^* = 0.1706
\end{cases}
&&
\begin{cases}
\|s\|_1 = S_\text{tot} = 3.9930 \\
\|\X\|_\infty^* = \ell_\text{max} = 0.0460.
\end{cases}
\end{align*}
Both the sufficient condition of Theorem~\ref{theo:taylor_expansion_complex}, and its modified $p$-norm version ($p=1$, $q=\infty$) are thus verified. Therefore, existence and uniqueness of a practical solution of the power flow equations is guaranteed.

In Figure~\ref{fig:voltage} we reported the true solution of the original nonlinear power flow equations (obtained numerically) together with the approximate linear solution proposed in Corollary~\ref{cor:approx}.
The average and maximum approximation errors are reported in Table~\ref{tbl:errors}.
In Figure~\ref{fig:voltage} we also reported the bounds \eqref{eq:buserror2} and \eqref{eq:buserror1} on the approximation error (as a solid and dashed curve, respectively). Both these bounds are quite conservative in the assessment of the quality of the approximation. However, they guarantee that the power flow equations have a unique solution in a region which is practical for the operation of the power distribution grid (other uninteresting solutions of the power flow equations may exist, at unacceptable voltage levels).

In order to better assess the quality of the approximated solution proposed in Corollary~\ref{cor:approx}, an extensive set of variations of the original testbed have been considered, including:
\begin{itemize}
\item different total grid load and load distribution
\item shunt capacitors for static voltage support
\item voltage regulation
\item different tap position at the PCC.
\end{itemize}
These simulations are available online \cite{github_approx-pf}, together with the source code used to generate them, and a selection of them is included hereafter.

In a first case, the active and reactive power demand of all loads has been doubled, in order to simulate a uniform overload of the grid. 
The uniform scaling of the vector of power demands $s$ affects in the same way (i.e. quadratically) both the approximation bounds \eqref{eq:buserror2} and \eqref{eq:buserror1}. While the bounds are still conservetive, Figure~\ref{fig:voltage_overload} shows that the approximation error is indeed larger compared to nominal case.

In a second case, the demand of one single bus (bus 32) has been increased to 2 MW of active power and 1 MVAR of reactive power, i.e. 50 times the original demand.
Figure~\ref{fig:voltage_spotoverload} shows how the approximation error is affected, similarly to the uniform overload case.
However, in this case, the effect of the overload on the approximation bounds \eqref{eq:buserror2} and \eqref{eq:buserror1} is qualitatively different, as different norms of $s$ are involved in the two bounds.
This example also shows that it is not possible, for a given grid, to know a-priori which particular $p$-norm yields the tightest bounds.

The approximation errors for these two cases are reported in Table~\ref{tbl:errors}, for comparison with the nominal case.

\begin{table}
\centering
\ra{1.2}
\addtolength{\belowbottomsep}{1mm}
\begin{tabular}{@{}lcccc@{}}
\toprule
& \multicolumn{2}{c}{absolute error} & \multicolumn{2}{c}{relative error${}^\dagger$} \\
\cmidrule{2-3} \cmidrule{4-5}
& avg. & max & avg. & max \\
\midrule
\multicolumn{5}{@{}l}{\textbf{Nominal case} (Figure \ref{fig:voltage})} \\
~~Voltage magnitude [p.u.] & 0.0041 & 0.0056 & 7.88\% & 8.45\% \\
~~Voltage angle [deg] & 0.0097 & 0.0178 & 0.43\% & 0.66\% \\
\multicolumn{5}{@{}l}{\textbf{Uniform overload} (Figure \ref{fig:voltage_overload})} \\
~~Voltage magnitude [p.u.] & 0.0191 & 0.0261 & 16.72\% & 17.94\% \\
~~Voltage angle [deg] & 0.0999 & 0.1782 & 2.09\% & 3.02\% \\
\multicolumn{1}{@{}l}{\textbf{Lumped overload} (Figure \ref{fig:voltage_spotoverload})} \\
~~Voltage magnitude [p.u.] & 0.0197 & 0.0373 & 18.99\% & 21.59\% \\
~~Voltage angle [deg] & 0.0994 & 0.3112 & 2.12\% & 4.27\% \\
\bottomrule
\end{tabular}
\begin{minipage}{0.93\columnwidth}
{\scriptsize ${}^\dagger$ The relative error for the voltage magnitude and angle at each bus $h$ is computed with respect to the voltage drop $V_0 - |v_h|$ and to the angle difference $|\theta_0 - \angle{v_h}|$, respectively.}
\end{minipage}
\caption{Absolute and relative approximation errors}
\label{tbl:errors}
\end{table}

\section{Approximation in polar coordinates: PV nodes and a comparison with the DC power flow model}
\label{sec:dcpower}

In this section, we show how to translate the model proposed in Corollary~\ref{cor:approx} to polar coordinates, i.e. in terms of voltage magnitudes and phases. 

\subsection{Voltage magnitudes}
\label{ssec:pvbuses}

If we denote by the symbol $|y|$ the vector having as entries the magnitudes of the entries of a complex vector $y$, then from the approximate model proposed in Corollary~\ref{cor:approx} we can obtain
$$
|\hat v_\lnodes|
= V_0 \left|\1 + \frac{1}{V_0^2} \X \bar s_\lnodes \right| .
$$
If we assume that that $\|\X \bar s_\lnodes\| / V_0^2 \ll 1$, i.e. that the voltage drops are much smaller than the nominal voltage, then from the fact that $|1+a|\approx 1+\realpart(a)$ for $|a| \ll 1$, we obtain 
\begin{equation}
\label{eq:voltagemag}
|\hat v_\lnodes'|
= \1 V_0 + \frac{1}{V_0} \realpart(\X \bar s_\lnodes).
\end{equation}
The approximation \eqref{eq:voltagemag} gives the opportunity to understand how the proposed model can be useful in a more general context, such as when there are voltage regulated (PV) buses in the grid.
Let $\vnodes$ be the subset of PV buses, whose voltage magnitude is regulated to $|v_{\vnodes}| = \eta$, and let $\pnodes = \lnodes \backslash \vnodes$ be the remaining set of PQ buses.
Let then $\X$ be divided into the corresponding blocks
$$
\X = \begin{bmatrix}
\X_{\vnodes \vnodes} 
& \X_{\vnodes \pnodes} \\
\X_{\pnodes \vnodes}
& \X_{\pnodes \pnodes} \\
\end{bmatrix}.
$$
Therefore, approximation \eqref{eq:voltagemag} for the nodes in $\vnodes$ becomes
\begin{align*}
\eta &= \1 V_0 + \frac{1}{V_0} 
\Big(
\realpart(\X_{\vnodes \vnodes} \bar s_{\vnodes}) + 
\realpart(\X_{\vnodes \pnodes}   \bar s_{\pnodes})
\Big) \\
&= \1 V_0 + \frac{1}{V_0} 
\Big(
\realpart(\X_{\vnodes \vnodes}) p_\vnodes + \imagpart(\X_{\vnodes \vnodes}) q_{\vnodes}\\
&\qquad\qquad\qquad\qquad+ \realpart(\X_{\vnodes \pnodes}) p_\pnodes + \imagpart(\X_{\vnodes \pnodes}) q_{\pnodes}
\Big).
\end{align*}
This linear relation can be inverted, in order to obtain the following expression for the reactive power injection of the nodes in $\vnodes$
\begin{multline*}
q_{\vnodes} = -\imagpart(\X_{\vnodes\vnodes})^{-1} 
\Big(
V_0 (V_0 \1 - \eta) + 
\realpart(\X_{\vnodes \vnodes}) p_\vnodes \\
+ \realpart(\X_{\vnodes \pnodes}) p_\pnodes 
+ \imagpart(\X_{\vnodes \pnodes}) q_{\pnodes}
\Big).
\end{multline*}
Plugging this expression in the model proposed in Corollary~\ref{cor:approx} yields an approximate solution of the power flow equations that is a linear function of the active and reactive power of the the PQ buses, and of the active power and voltage magnitude set-points of the PV buses. The quality of this approximated model has been verified numerically on the same testbed of Section~\ref{sec:simulations}, by assuming that buses 15 and 51 are voltage regulated to $\eta = V_0$.
Also this simulation is available in \cite{github_approx-pf}, and the results are reported in Figure~\ref{fig:voltage_pvbus}.

\begin{figure*}[t]
\centering
\includegraphics[scale=1.06]{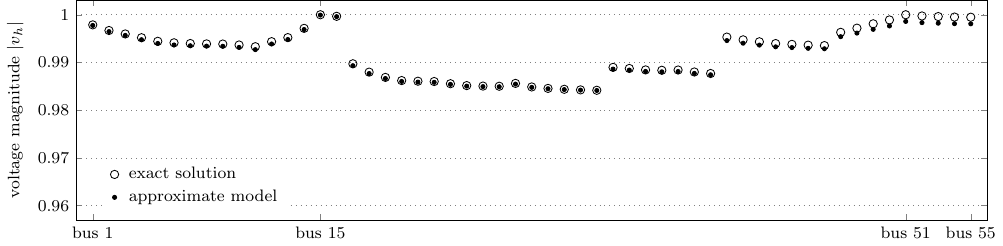}
\caption{Bus voltages in the modified IEEE 123 test feeder \cite{github_approx-pf}, in the case of two voltage regulated buses (bus 15 and bus 51). The voltage reference for these two buses has been set to 1 p.u. The circles represent the exact solution of the nonlinear power flow equations, while the dots represent the approximate solution $\hat v$ given by \eqref{eq:approximate}, where the two PV buses have been included as described in Section~\ref{ssec:pvbuses}.}
\label{fig:voltage_pvbus}
\end{figure*}

\subsection{Voltage phases}

We now consider the phases of the approximate power flow solution, and we show how the proposed approximation can be seen as a generalization of the DC power flow model. 
If we denote by $\hat \theta_\lnodes := \phase{\hat v_\lnodes}$ the vector having as entries the phases of the entries of the approximate solution $\hat v_\lnodes$, 
we have
\begin{equation}
\label{eq:THETAmodel}
\hat \theta_\lnodes
= \slackangle \1 + \phase{\1 + \frac{1}{V_0^2} \X \bar s_\lnodes}.
\end{equation}
By using the fact that $\phase{1+a}\approx \imagpart(a)$ for $|a| \ll 1$, we obtain the intermediate approximated model
\begin{equation}
\label{eq:LINmodel}
\hat \theta_\lnodes' := \slackangle \1 + \frac{1}{V_0^2} \imagpart\left( \X \bar s_\lnodes\right).
\end{equation}
In the special case in which power lines are assumed to be purely inductive, namely $\X = jX$, one obtains 
\begin{equation}
\label{eq:dcpowerflow}
\hat \theta_\lnodes' = \slackangle \1 + \frac{1}{V_0^2} X p_\lnodes,
\end{equation}
Equation \eqref{eq:dcpowerflow} is known in the literature as DC power flow model (see \cite{Stott2009} and references therein) which is based exactly on the assumption of small voltage drops and purely inductive lines. Equation \eqref{eq:dcpowerflow} is typically shown in the following equivalent form
$$
B \theta \approx \frac{1}{V_0^2} p
$$
where $\theta$ and $p$ are the vectors in $\C^{n+1}$ with entries $\theta_\node$ and $p_\node$, where $B$ is the power lines susceptance matrix (i.e. a real matrix such that $Y = j B$), and where $V_0$ is the slack voltage (the voltage magnitude at the PCC).

These different approximations of the bus voltage angles have been plotted in Figure~\ref{fig:angles} for the same testbed. It is clear how the model proposed in Corollary~\ref{cor:approx} approximates very well the true voltage angles. On the other hand, the DC power flow model presents a much larger approximation error, due to the inaccurate assumption that power lines are purely inductive. Indeed, the model \eqref{eq:LINmodel}, which differs from the DC power flow model only for the fact that this assumption has not been made, provides a much more accurate approximation.

In light of this analysis, the model in Corollary~\ref{cor:approx} can also be interpreted as a generalization of the DC power flow model to the case of generic power line impedances.

\begin{figure*}[tb]
\centering
\includegraphics[scale=1.06]{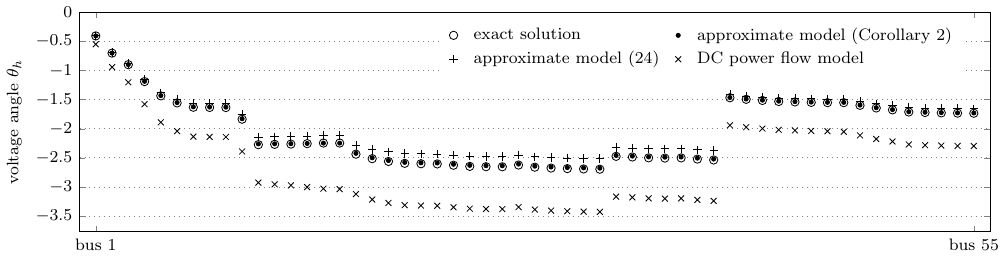}
\caption{Bus voltage angles in the modified IEEE 123 test feeder \cite{github_approx-pf}, according to the different approximations analyzed in Section~\ref{sec:dcpower}. The circles represent the true solution of the nonlinear power flow equations. The dots represent the proposed approximate solution proposed in Corollary~\ref{cor:approx} and given by \eqref{eq:THETAmodel}. The $+$ signs represent the intermediate model \eqref{eq:LINmodel}, where power lines have not been assumed to be purely inductive. The $\times$ signs represent the DC power flow model.}
\label{fig:angles}
\end{figure*}

\section{Conclusion}
\label{sec:conclusions}

In Theorem~\ref{theo:taylor_expansion_complex} we derived a sufficient condition for the existence of a practical solution to the nonlinear power flow equations that describe a power distribution network.
This condition is proved to be also necessary in at least one example, and is verified in the IEEE testbed that we adopted for numerical validation.
Different variations of this condition can be obtained by adopting different norms (and thus different invariant sets for the Banach fixed point theorem), yielding various physical interpretations. 

In the case of a grid of constant power (PQ) buses and one slack bus, the existence result immediately returns the approximate solution of the power flow equations presented in Corollary~\ref{cor:approx}. Such approximation is conveniently linear in the active and reactive power references of the buses, and analytical bounds for its error are given.
We showed that, via some manipulations, it is also possible to consider different bus models, and in particular voltage regulated (PV) buses.

We finally assessed the quality of the proposed model via numerical simulations on a set of variations of the IEEE 123 test feeder, showing how it outperforms the classical DC power flow model. The proposed model has the potential of serving as a flexible tool for the design of control, monitoring, and estimation strategies for the power distribution grid.

\appendix

\section{}
\label{app:norms}

\begin{lemma}
Let $\|x\|_p$ be the $p$-norm of a complex vector as defined in \eqref{eq:vectorpnorm}, and $\|A\|^*_q$ be the matrix norm defined in \eqref{eq:matrixpnorm}. Assume that $p,q \in [1,\infty) \cup \{\infty\}$ are such that $1/p+1/q=1$. Then
$$
\max_{\substack{\|x\|_p=1\\\|y\|_p=1}} 
\left\| \diag(x) A y \right\|_p = \|A\|^*_q.
$$
\label{lem:pnorm}
\end{lemma}

\begin{proof}
We first prove the lemma in one direction. We have that, if $z:=\diag(x)Ay$, then
\begin{multline*}
|z_\node|
=\left|  x_\node \sum_\nodealt A_{\node\nodealt} y_\nodealt \right|
=|x_\node|\ \left| A_{\node\bullet} y \right| \le \\
\le |x_\node|\  \| A_{\node\bullet}\|_q\  \|y\|_p
\le |x_\node |\  \| A\|^*_q \ \|y\|_p 
\end{multline*}
where we have applied Holder's inequality.
Hence, in case $p<\infty$ we have that
\begin{multline*}
\|z\|_p
=\left(\sum_\node |z_\node|^p \right)^{1/p}
\le \left( \sum_\node \left(
|x_\node | \ \| A\|^*_q \ \|y\|_p 
\right)^p
\right)^{1/p}\\
=\|x\|_p\ \|A\|^*_q \  \|y\|_p.
\end{multline*}
In case $p=\infty$ we have that
\begin{multline*}
\|z\|_\infty
=\max_\node |z_\node |
\le \max_\node |x_\node |\ \| A\|^*_1 \ \|y\|_\infty =\\
\|x\|_\infty \ \|A\|^*_1 \  \|y\|_\infty.
\end{multline*}

In order to prove the other direction first notice that it is well known that Holder's inequality is tight in the sense that, for any fixed complex vector $a$ there exists a complex vector $b$ with the same dimension such that $\|b\|_p=1$ and $|\sum_k a_k b_k|=\|a\|_q$. From this fact, if $h'$ is the index such that $\| A\|^*_q=\| A_{h'\bullet}\|_q$, we let $x=e_{h'}$ (namely the $h'$-th vector of the canonical base) and $y$ be a vector such that $\|y\|_p=1$ and
$|\sum_k A_{h' k} y_k |=\| A_{h'\bullet}\|^*_q$. With this choice it is easy to verify that
$$\|\diag(x)Ay\|_p=\| A\|^*_q$$
\end{proof}

\begin{lemma}
Assume that we have a vector norm $\|\cdot\|$ and a matrix norm $\|\cdot\|^*$ such that $\|\diag(x) A y\| \le \|A\|^*\|x\|\|y\|$ for any $x$, $y$.
If we define the function
$$
F(x):=\diag(x+a)A(\bar x+\bar a)
$$
then
$$
\|F(x_1)-F(x_2)\| \le \|A\|^* \big(\|x_1+x_2\|+2\|a\|\big)
\|x_1-x_2\|.
$$
\label{lem:F}
\end{lemma}
\begin{proof}
First observe that
\begin{align*}
\|F(x_1)-F(x_2)\|
& = \|\diag(x_1)A\bar x_1+\diag(x_1)A\bar a\\
& \quad + \diag(a)A\bar x_1-\diag(x_2)A\bar x_2\\
& \quad -\diag( x_2)A\bar a-\diag( a)A\bar x_2\|\\
& \le \|\diag( x_1)A\bar x_1-\diag( x_2)A\bar x_2\|\\
& \quad +\|\diag( x_1- x_2)A\bar a\|\\
& \quad +\|\diag( a)A(\bar x_1-\bar x_2)\|.
\end{align*}
Notice that
\begin{multline*}
\diag( x_1)A\bar x_1-\diag( x_2)A\bar x_2 \\
=\frac{1}{2}(\diag( x_1- x_2)A(\bar x_1+\bar  x_2) + \diag( x_1+ x_2)A(\bar x_1- \bar x_2))
\end{multline*}
and so
$$
\|\diag( x_1)A\bar x_1-\diag( x_2)A\bar x_2\|\le \|A\|^* \|x_1-x_2\| \|x_1+x_2\|.
$$
From this we can argue that
\begin{align*}
\|F(x_1)-F(x_2)\|&\le \|A\|^* \|x_1-x_2\| \|x_1+x_2\|\\
& \quad + 2\|A\|^* \|a\| \|x_1-x_2\|\\
& = \|A\|^*(\|x_1+x_2\|+2\|a\|)\|x_1-x_2\|.
\end{align*}
\end{proof}

\begin{lemma}\label{lemma}
Let $x,a,b\ge 0$ such that $ab\le 1/4$. Then $x=4ab^2$
satisfies 
\begin{equation}\label{second}
a(x+b)^2\le x
\end{equation} 
\end{lemma}

\begin{proof}
To prove the lemma it is enough to substitute $x=4ab^2$ in (\ref{second}) and to verify that the inequality holds. Doing this we need to verify that 
$$4ab^2\ge a(4ab^2+b)^2=ab^2(4ab+1)^2$$
This inequality holds if and only if $4\ge (4ab+1)^2$, which is true since by hypothesis we have that
$0\le 4ab\le 1$.
\end{proof}

\section*{Acknowledgment}
\addcontentsline{toc}{section}{Acknowledgment}

The authors thank Prof. Florian D\"orfler for his insightful and useful comments on the first draft of the paper.



\end{document}